\documentclass{article}
\usepackage{latexsym}
\usepackage{float}
\usepackage{pstricks}
\usepackage{amssymb}
\usepackage{amsmath} 
\usepackage{amsfonts}
\usepackage{mathrsfs}
\usepackage{mathtools}

\def\noi{\noindent}
\def\Id{\text{Id}}
\def\til{\widetilde}
\mathtoolsset{showonlyrefs}
\usepackage{amsthm,array}
\usepackage{array}
\usepackage{cite}
\usepackage{graphicx}
\usepackage{multirow}

\newtheorem{theorem}{Theorem}[section]

\newtheorem{lemma}[theorem]{Lemma}
\newtheorem{corollary}[theorem]{Corollary}

\newtheorem{definition}[theorem]{Definition}
\newtheorem{remark}[theorem]{Remark}

\newtheorem{remark-definition}[theorem]{Remark and Definition}

\begin{document}
	
	\title{Guaranteed estimates for the length of branches of periodic orbits for equivariant Hopf bifurcation}
	
	\author{\setcounter{footnote}{3}
		E. Hooton\footnote{Institute of Mathematics, Czech Academy of Sciences, \v{Z}itn\'a 25, 11567 Prague, Czech Republic.}, \setcounter{footnote}{0}
		Z. Balanov\footnote{Department of Mathematics, Xiangnan University, Chenzhou, Hunan 423000, China} \footnote{Department of Mathematical Sciences, The University of Texas at Dallas, USA.}
		\footnote{Corresponding author; email: balanov@utdallas.edu},
	D. Rachinskii$^\dagger$
}
	\maketitle
	
	\begin{abstract}
		Connected branches of periodic orbits originating at a Hopf bifurcation point of a differential system are considered. A computable estimate for the range of amplitudes of periodic orbits contained in the branch is provided under the assumption that the nonlinear terms satisfy a linear estimate in a ball. If the estimate is global, then the branch is unbounded. The results are formulated in an equivariant setting where the system can have multiple branches of periodic orbits characterized by different groups of symmetries. The non-local analysis is based on  the equivariant degree method, which allows us to handle both generic and degenerate Hopf bifurcations. This is illustrated by examples.
	\end{abstract}



\section{Introduction}

The theorem of Rabinowitz \cite{rabinowitz}, extending the work of Krasnosel'skii \cite{krasnoselskii}, provides a classical topological tool for analysis of global bifurcations.
It establishes that a continuous branch of fixed points bifurcating from a trivial solution either extends to infinity or connects back to the trivial solution at another bifurcation point. 
The global bifurcation theorem of Alexander and Yorke \cite{alexanderyorke} establishes a counterpart of this result for branches of periodic orbits bifurcating from an equilibrium via a Hopf bifurcation (see also \cite{malletparetchow, FiedlerBook2,IV-B,AED,Wu1}). The theorem states that these branches 
are either unbounded or connect to an equilibrium at another Hopf bifurcation point. 
The latter of the two alternatives can sometimes be excluded by local analysis at the equilibrium point, in which case the theorem guarantees the existence of an unbounded branch of periodic orbits.

It is important to note that in the above theorems branches of periodic orbits are considered in Fuller space, i.e.~in the product of the 
space of periodic functions (solutions) and the  space of parameters which include the bifurcation parameter $\alpha$ and the unknown period $p$ of periodic orbits \cite{qingwen}.
Thus, a branch is unbounded if it contains elements for $\alpha$ arbitrarily close to the boundary of the parameter interval, or contains periodic orbits of arbitrarily large amplitude (norm $\|x\|$), or orbits of arbitrarily large period $p$, or several of these possibilities are combined. 
In particular, it is a non-trivial problem to determine whether the branch is unbounded because it contains large-amplitude periodic orbits or their periods are unbounded.  An example of the latter possibility is a branch connecting a Hopf bifurcation point with a homoclinic bifurcation point. In this context, {\em a priori} estimates of the period play an important role. However, they are typically hard to establish, and 
the question is further complicated by the fact that periods considered in the theorem are not necessarily the minimal period.
If the norm of periodic orbits along the branch is uniformly bounded, the above topological results do not provide any estimate for this norm. 

The purpose of this paper is to prove the existence of continuous branches that contain periodic orbits of all amplitudes ranging from zero to $\|x\|=R$ where we can explicitly control $R$. In particular, under certain conditions, this branch is unbounded and $\|x\|$ ranges from zero to infinity, while the minimal period of the orbits is unifomly bounded. We prove the existence of a non-stationary periodic solution of any given norm  $\|x\|=s$ satisfying $0< s<R$ by showing that the equivariant $S^1$-degree of a vector field $\mathcal F$ associated with the problem is non-trivial on the boundary of some domain $\Omega$ containing the periodic orbit. The proof is completed by  Kuratowski's lemma which ensures in a standard way that  all these solutions are embedded in a connected branch of non-stationary periodic orbits stemming from a Hopf bifurcation point. 
Because the periodic orbits of interest are neither small nor close to the bifurcation point, the analysis is non-local, and the domain $\Omega$ 
is designed in a special way. In order to compute the $S^1$-degree, we demand 
the vector field $\mathcal F$ to have a principal linear (with respect to $x$) component $a(\alpha,p)x$ on a part of the boundary of $\Omega$, in the sense that a certain projection $\mathcal Q \mathcal F$ of the vector field satisfies $\|\mathcal Q\mathcal F(\alpha,p,x) - a(\alpha,p)x\|\le \|a(\alpha,p)x\|$. This condition limits the size of the domain $\Omega$ and the maximal amplitude $\|x\|=R$ of the orbits that we can capture, unless the above estimate is global.

In addition, we prove the theorem in the equivariant setting where the system respects a group of spatial symmetries $\Gamma$. In other words, we consider a $\Gamma$-equivariant Hopf bifurcation. As a typical scenario, this bifurcation can give rise to multiple branches of periodic orbits characterized by different groups of spatio-temporal symmetries. In order to ensure the existence of a branch with a specific symmetry, and estimate the range of norms of its elements, we restrict the opeartor of the problem to the fixed point space of the corresponding symmetry group and apply the non-local topological construction described above to the restricted operator. As we illustrate by examples, this approach allows us to handle not only generic $\Gamma$-equivariant systems but also a number of degenerate cases. One of them is the simultaneous Hopf and steady state bifurcations. In another degenerate resonance situation the crossing number is undefined because the linearization has a pair of purely imaginary roots for all the parameter values, and other roots cross the imaginary axis through this pair at the Hopf bifurcation point. 


The paper is organized as follows. Section \ref{prelim} contains a brief account of 
{the $S^1$-degree \cite{Dylawer}, which is the main equivariant topological tool used in the proofs 
	(see also \cite{AED} for the axiomatic approach to the $S^1$-degree and \cite{AED,IV-B} for a systematic exposition of the equivariant degree theory and its applications to symmetric Hopf bifurcation).}
The main result and its proof are presented in Sections \ref{main} and \ref{proof}. Section \ref{examples} contains three examples. Some notation of the latter section is explained in Appendix.


\section{Preliminaries}\label{prelim}

In this section, we provide some equivariant degree background. 


\subsection{$S^1$-degree}\label{subsec:S1-degree}
Let $G$ be a compact Lie group acting on a metric space $X$ (see, for example, \cite{AED,Brocker-tomDieck}). For any $x \in X$, put $G(x)=\{gx \in X \; : \; g \in G\}$ and call it the orbit of $x$.
A set $Z \subset X$ is called $G$-invariant (in short, invariant) if it contains all its orbits. Given a (closed) subgroup $H \leq G$, denote by $X^H:= \{x \in X \, : \, hx = x \; \forall h \in H\}$ the set of all $H$-fixed points of $X$.
Assume $G$ acts on two metric spaces $X$ and $Y$. A continuous map $f : X \to Y$ is called $G$-equivariant
if $f(gx) = gf(x)$ for all $x \in X$ and $g \in G$. In particular, if the action of $G$ on $Y$ is trivial, then the equivariant map is called $G$-invariant. We refer to \cite{Brocker-tomDieck,GolSteShef,GS,AED} for the equivariant topology and representation theory background frequently used in the present paper.

Let $V$ be an orthogonal $S^1$-representation. Suppose that an open bounded set $\Omega \subset \mathbb R \oplus V$ is invariant with respect to the $S^1$-action, where we assume that $S^1$ acts trivially on $\mathbb R$. As it is well-known, for any $x \in \Omega$, one has $G_x = S^1$ or $G_x = \mathbb Z_k$, where $\mathbb Z_k$ stands for a cyclic group of order $k$.

We say that an equivariant map $f: \overline{\Omega} \to V$ is admissible if $f^{-1}\{0\} \cap \partial \Omega = \emptyset$. In this case,  $(f,\Omega)$ is called an admissible pair. Similarly, a continuous map $h : [0,1] \times \overline{\Omega} \to V$ is called an admissible (equivariant) homotopy if $h(t,\cdot)$ is admissible for any $t \in [0,1]$. It is possible to axiomatically define a unique function $S^1$-deg which assigns to each admissible pair a finite linear combination   $\sum_{k=1}^m n_{l_k}(\mathbb Z_{l_k})$, where $n_{l_k} \in \mathbb Z$
(cf.~\cite{AED}, pp. 109, 113). The following is a partial list of the axioms:

\medskip

\noi\textbf{(A1) (Existence)} If $S^1\text{\rm -deg}\,(f,\Omega) = \sum_{k=1}^m n_{l_k}(\mathbb Z_{l_k})$ and $n_{l_k} \neq 0$ for some $k$, then there exists an $x\in \Omega$ such that $f(x) = 0$ and
$\mathbb Z_{l_k} \subset G_x$.

\medskip

\noi\textbf{(A2) (Homotopy)} If $h:[0,1]\times \overline{\Omega} \to V$ is an admissible equivariant homotopy, then the value of
$S^1\text{\rm -deg}\,(h(\mu,\cdot),\Omega)$ is the same for each $\mu$.

\medskip

\noi\textbf{(A3)(Additivity)} For two invariant open disjoint subsets $\Omega_1,\Omega_2 \subset \Omega$ with $f^{-1}(0)\cap \Omega\subset  \Omega_1\cup\Omega_2 $, one has 
$$S^1\text{\rm -deg}\,(f,\Omega)= S^1\text{\rm -deg}\,(f,\Omega_1 )+ S^1\text{\rm -deg}\,(f,\Omega_2).$$

\medskip
\noi
\textbf{(A4)(Normalization)} Denote by  $\mathcal V_1$ the complex plane equipped with the $S^1$-action induced by complex multiplication:
$\gamma z := \gamma \cdot z$, $\gamma= e^{i\theta} \in S^1$, $z \in \mathbb C$.
Define a set $\Omega_0$ and a map $b : \mathbb R \oplus \mathcal V_1 \to \mathcal V_1$ by 
$$
\Omega_0 = \Big\{(t,z) \in \mathbb R \oplus \mathcal V_1 \; : \; |t| < 1, \;\; {1 / 2} < \|z\| < 2\Big\}, \quad b(t,z)= (1 - \|z\| + it)\cdot z.
$$
Then, $S^1\text{\rm -deg}\,(b,\Omega_0)=  1 \cdot (\mathbb Z_1)$.

\medskip
\noi
{\textbf{(A5)(Multiplicativity)} Suppose that $\mathcal U$ is a finite-dimensional space with the trivial $S^1$-representation,  $U$ is an open bounded neighborhood of zero in $\mathcal U$ and $g : U \to \mathcal U$ is a continuous map with no zeros on $\partial U$ . Then,  
	$$S^1\text{\rm -deg}\,(f \times g, \Omega \times U) = S^1\text{\rm -deg}\,(f,\Omega) \cdot \deg(g,U),$$
	where ``$\deg$" stands for the Brouwer degree.}

\medskip
\noi
\textbf{(A6)(Suspension)} Suppose that $\mathcal U$ is an orthogonal $S^1$-representation and $U$ is an open bounded invariant neighborhood of zero in $\mathcal U$. Then, $$S^1\text{\rm -deg}\,(f \times {\rm Id}, \Omega \times U) = S^1\text{\rm -deg}\,(f,\Omega).$$

{\begin{remark}\label{rem:excision}{\rm 
			In a standard way, using property {\bf (A3)}, one can deduce the {\it excision} property of the $S^1$-degree. Namely, if $f^{-1}(0) \cap \Omega \subset \Omega_0$, where
			$\Omega_0 $ is an invariant open subset of $\Omega$, then  $S^1\text{\rm -deg}\,(f ,\Omega) = S^1\text{\rm -deg}\,(f ,\Omega_0)$.}
\end{remark}}

{Combining the equivariant version of the standard Leray-Schauder projection with property {\bf (A6)},} one can define the $S^1$-degree for $S^1$-equivariant compact vector fields (see \cite{AED, IV-B} for details). Also, combining the axioms of the $S^1$-degree with some standard homotopy theory techniques, one can reduce the computation of the $S^1$-degree of the maps naturally associated with a system undergoing the Hopf bifurcation to the computation of the Brouwer degree. To be more precise,  let $V$ be an orthogonal
$S^1$-representation with $V^{S^1}=\{v\in V:\ (\gamma, v)= v \ \ \forall \gamma \in S^1\}=\{0\}$.  Take the isotypical decomposition
$$
V = V_{k_1} \oplus V_{k_2} \cdots \oplus V_{k_s},$$
where each $V_{k_j}$ is modeled by the $k_j$-th irreducible representation. Define
\begin{equation}\label{eq:mathcal-O}
\mathcal O=\{(\lambda,v) \in \mathbb C \oplus V:\|v\|<2\ , {1}/{2}<|\lambda|<4\}.
\end{equation}
Now, consider a map $a:S^1 \to GL^{S^1}(V)$ and define $a_j:S^1 \to GL^{S^1}(V_{k_j})$ by 
the formula  $a_j(\lambda)= a(\lambda)_{|V_{k_j}}$ (see, \cite[p.~284]{AED}). Let $f_a:\overline{\mathcal O} \to \mathbb R\oplus V$ be an
$S^1$-equivariant map defined by
\begin{equation}\label{eq:f-a}
f_a(\lambda,v)= \bigl(|\lambda|(\|v\|-1)+\|v\|+1, a\left({\lambda}/{|\lambda|}\right)v \bigr).
\end{equation}
The following formula {(combined with Property {\bf (A5)}}) plays an important role in our proofs:
\begin{equation}\label{eq:main-formula}
S^1\text{\rm -deg}\,(f_a,\mathcal O)= \sum_{j=1}^s \left( \deg(\text{det}_{\mathbb C}\circ a_j,B)\right)(\mathbb Z_{k_j}),
\end{equation}
where $B$ stands for the unit ball in $\mathbb C$ (cf. \cite{AED}, Theorem 4.23).

\subsection{
	Spatio-temporal symmetries of periodic functions}
	\label{subsec:twisted-subgr}
{If $\Gamma$ is a finite group and $W$ is a  $\Gamma$-representation, then for any periodic function $x :\mathbb R \to W$, its 
spatio-temporal symmetries are described by a subgroup $ H <  \Gamma$ and a homomorphism $\varphi : H \to S^1 =  \mathbb R / \mathbb Z$. This information is encoded in the graph of the homomorphism $\varphi$ which we will denote by $H^{\varphi}$. To be more specific, if $x$ is a $p$-periodic function with symmetry group $H^{\varphi}$, then for each $h \in H$, one has $h x (t - \varphi(h)p ) = x(t)$ for any $t$.  Clearly, if $x$ is a non-constant function, then 
$H^\varphi $ is a finite group. Several twisted subgroups important for the present paper are explicitly described in the Appendix.}

%

\section{Main result}\label{subsec:main-results}\label{main}
Let $\Gamma$ be a finite group and $V=\mathbb{R}^N$ an orthogonal  $\Gamma$-representation. 
Suppose $A: [\alpha_-,\alpha_+] \to L^{\Gamma}(V;V)$ is a continuous curve in the space of $\Gamma$-equivariant linear maps from $V$ to $V$ and $f: [\alpha_-,\alpha_+]\times V\to V$ is a continuous $\Gamma$-equivariant map (we assume that $\Gamma$ acts trivially on $[\alpha_-,\alpha_+]$). We are interested in the existence of branches of periodic solutions with a prescribed spatio-temporal symmetry for the  equation
\begin{equation}\label{PNH}
\dot{x}= A(\alpha)x  + f(\alpha, x),\qquad x\in V.
\end{equation}
Further, we are interested in effective estimates of the length of these branches.
To be more precise, following the standard scheme based on the normalization of the period (see, for example, \cite{AED}), instead of looking for $p$-periodic solutions to \eqref{PNH} with unknown 
period $p$, one can introduce $p$ as an additional parameter and reduce the original problem to looking for $2\pi$-periodic solutions.
To this end, put $\beta:= {2\pi / p}$ and apply the change of variables
$$
u(t) = x\Big(\frac{pt}{2\pi} \Big)
$$
to obtain the problem
\begin{equation}\label{eq:per-normalization}
\begin{cases}
\beta \dot{u} = A(\alpha)u + f(\alpha,u),\\
u(0) = u(2\pi).\\
\end{cases}
\end{equation}
Denote 
by $C:= C(S^1;V)$ the space of continuous $V$-valued maps on $S^1$ equipped with the sup-norm.
We naturally identify $2\pi$-periodic functions with the elements of $C$.
\begin{definition}\label{YTM}
{Let $\mathcal{B}$ be  a set of non-constant solutions $(\alpha, \beta, x(t)) \subset [\alpha_-,\alpha_+]\times[\beta_-,\beta_+]\times C$ of
problem \eqref{eq:per-normalization}  such that $[\beta_-, \beta_+] \subset (0, \infty)$. The set $\mathcal{B}$ is called a branch joining the spheres of radius $r$ and $R$ centered at the origin of $C$ if  $\mathcal{B}$ is a compact connected set in the space $[\alpha_-,\alpha_+]\times[\beta_-,\beta_+]\times C$
equipped with the product norm and $\mathcal{B}$
has a non-empty intersection with each of the sets $[\alpha_-,\alpha_+]\times[\beta_-,\beta_+]\times \{\|x\|_C=r\}$ and $[\alpha_-,\alpha_+]\times[\beta_-,\beta_+]\times \{\|x\|_C=R\}$.}
\end{definition}

{Denote $G:=\Gamma \times S^1$. To connect symmetric properties of the branches with equivariant spectral properties of $A(\alpha)$, 
denote by  $\til{V}=\mathbb{C}^N$ the complexification of the representation $V$ and extend
the complex $\Gamma$-representation $\til V$} to a real $G  $-representation  $^l\til V$ by defining the $l$-folded action of $S^1$ by $e^{i{\theta}}\cdot v := e^{il{\theta}}v$. 
The following family of finite-dimensional maps $\Delta_l(\alpha, \tau, \beta) \in L^G(^l\til V; ^l\til V)$ will {play an important role in our considerations
{(here  $L^G(^l\til V; ^l\til V)$ stands for $G$-equivariant linear operators in $^l\til V$)}: 
\begin{equation}\label{eq:without}
\Delta_l(\alpha, \tau, \beta) :=  l(\tau +i\beta)\Id - A(\alpha).
\end{equation}
Further, take a twisted subgroup $H^\varphi < G$ (cf. Subsection \ref{subsec:twisted-subgr}) 
and denote by $^l\til V^{H^\varphi}$ the fixed point space of $H^\varphi$ and by $\Delta^{H^\varphi}_l(\alpha, \tau, \beta)$ the restriction of $\Delta_l(\alpha,\tau, \beta)$ to $^l\til V^{H^\varphi}$.
With this restriction, we associate
the map $\Lambda^{H^\varphi}_l:[\alpha_-,\alpha_+] \times \mathbb R \times \mathbb R \to \mathbb C$ defined by 
\begin{equation}\label{eq:det-Lambda}
\Lambda^{H^\varphi}_l (\alpha, \tau, \beta) := \begin{cases} \text{det}_\mathbb R (A(\alpha)|_{V^H}) \;\;\,\quad\qquad \text{if} \quad l=0, \\
\text{det}_{\mathbb C}(\Delta^{H^\varphi}_l(\alpha,\tau, \beta)) \qquad \text{otherwise},
\end{cases} 
\end{equation}
which characterizes symmetric properties of branches of  periodic solutions. For a fixed $\alpha$,\ the map $\Lambda^{H^\varphi}_l$
can be identified with a polynomial of the complex variable $\tau+i\beta$.

{To estimate the length of a branch,} 
the following quantity (which is the operator norm  of an operator acting from $L_2$ to $C$ associated with the periodic problem) will be used:  
\begin{equation}\label{linear_bound}
M^{H^\varphi}(\alpha, \beta):= \Big(\sum_{l=0}^\infty \left|(\Delta^{H^\varphi}_l(\alpha, 0, \beta))^{-1}\right|^2\Big)^{1/2},
\end{equation}
where $|\cdot|$ is the matrix norm induced by the norm in $V$.

Given a set $\mathcal P \subset [\alpha_-,\alpha_+] \times \mathbb R_+ \times \mathbb R_+$,
define
\begin{equation}\label{eq:P-0-set}
\mathcal{P}_{\pm} := \mathcal{P}\bigcap \left(\{\alpha_\pm\} \times \mathbb R_+ \times \mathbb R_+\right) \quad {\rm and} \quad
\mathcal P_0 := \mathcal{P}\bigcap \left([\alpha_-,\alpha_+] \times \{0\} \times \mathbb R_+\right),
\end{equation}
{where $\mathbb R_+$ denotes the non-negative semi-axis.}
We denote by $\partial \Omega$ the boundary of a domain ${\Omega}$
and by $\overline{\Omega}$ the closure of ${\Omega}$.

\medskip
{We make the following assumptions:}

\smallskip
\noi\textbf{(P0)} $A$ and $f$ in {\eqref{eq:per-normalization}} depend continuously on their arguments.

\smallskip
\noi\textbf{(P1)}  ${\Lambda^{H^\varphi}_0 (\alpha,0,0)}\neq 0 $ for all $\alpha\in [\alpha_-,\alpha_+]$. 

\smallskip
\noi\textbf{(P2)} There exists a bounded 
domain  $\mathcal{P} \subset [\alpha_-, \alpha_+] \times \mathbb R_+\times\mathbb R_+$ such that:

(i) $\mathcal{P}$  is homeomorphic to a closed ball;

\smallskip

(ii) $\Lambda^{H^\varphi}_1 (\alpha, \tau, \beta)\ne 0$ for all $(\alpha, \tau, \beta) \in \partial\mathcal{P}\setminus (\mathcal{P}_+ \bigcup \mathcal{P}_- \bigcup \mathcal P_0)$;

\smallskip

(iii) $\mathcal{P}_+$ and $\mathcal{P}_-$ contain a different number of roots of $\Lambda^{H^\varphi}_1 (\alpha, \tau, \beta)$ (counted according to their multiplicities).


\smallskip
\noi\textbf{(P3)}
There exists an {open}  set $\mathcal D \subset [\alpha_-,\alpha_+] \times {\{0\}  \times \mathbb R_+}$ such that

\smallskip
(i) $\overline{\mathcal D}$ is homeomorphic to a closed disk;

\smallskip

(ii) $(\Lambda^{H^\varphi}_1)^{-1}(0)\bigcap \overline{\mathcal D} = (\Lambda^{H^\varphi}_1)^{-1}(0)\cap \mathcal P_0$;

\smallskip

(iii) $\Lambda_l^{H^\varphi} (\alpha, 0, \beta) \neq 0$ for any $l\in \mathbb N$ and any $(\alpha,{0, \beta})\in \partial\mathcal D$.

\smallskip
\noi\textbf{(P4)} There exist $N(\alpha)$ and $0\le r<R$ such that for each $\alpha \in [\alpha_-,\alpha_+]$, 
\begin{equation}
\label{estimate}
|f(\alpha, x)|  \le N(\alpha) \max \{r,|x|\} \quad \text{for} \quad |x|\le R.
\end{equation}

\smallskip
\noi\textbf{(P5)} 
The following estimate holds: 
$$N(\alpha)< \frac{1}{\sqrt{2\pi}M^{H^\varphi}(\alpha,\beta)}
\quad \text{for} \quad(\alpha,0,\beta)\in \partial \mathcal D.
$$

\smallskip
\begin{remark}\label{rem:hypotheses} {\rm Condition \textbf{(P0)} is a mild regularity requirement. Condition \textbf{(P1)} guarantees the absence of steady state bifurcation. Assumption \textbf{(P2)}(iii) provides the non-triviality of the (isotypical) crossing number, while \textbf{(P3})(iii) is a  weak version of the non-resonance condition. Assumptions \textbf{(P4), (P5)} ensure that the vector field associated with the problem has a principal linear part on the set difference of balls of radii $R$ and $r$.}

{\rm 
It was observed by J. Ize \cite{Ize-Topolo-bifurcation} that the occurrence of the
Hopf bifurcation with prescribed symmetry is related to the non-triviality of the equivariant $J$-homomorphism associated with the equivariant operator equation. This observation gives rise to the following two questions: (a) Under which conditions on the right-hand side of \eqref{PNH} is the $J$-homomorphism correctly defined? (b) 
Under which conditions is this homomorphism non-trivial? From this viewpoint, conditions \textbf{(P0)}, \textbf{(P1)} and \textbf{(P3)} are related to (a), while condition \textbf{(P2)} is related to (b).} 
\end{remark}

{We are now in a position to formulate the main result of the present paper.}
\begin{theorem}
\label{thm:main-theorem}\label{t1}
Suppose conditions \textbf{(P0)--(P5)} are satisfied.
Then, there exists a branch of non-constant periodic solutions to system  \eqref{PNH} joining  the sphere of radius $r$ to the sphere of radius  $R$ (cf.~Definition \ref{YTM}). Solutions of this branch have {spatio-temporal} symmetry at least $H^\varphi$.
\end{theorem}

\begin{remark}{\rm
		It will be shown in the proof that the minimal period of the periodic solutions of the branch is uniformly bounded.
If $r=0$ in {\bf (P4)}, then the branch  connects to a Hopf bifurcation point. If $R=\infty$ (the estimate in {\bf (P4)} is global), then the branch extends to infinity. A non-equivariant variant of the theorem was proved in \cite{KR} for a class of scalar equations in which the nonlinearity satisfies a global sector estimate.
}	
\end{remark}

In \cite{AED}, the $\Gamma$-equivariant Hopf bifurcation was {studied} using an invariant known as the $\Gamma \times S^1$-equivariant twisted degree. Its values are finite linear combinations of the form 
\begin{equation}\label{eq:-equiv-degree}
\sum_i n_{\varphi_i} (H^{\varphi_i}),
\end{equation}
where $n_{\varphi_i} \in \mathbb Z$ and $(H^{\varphi_i})$ is a twisted orbit type. Generically, the coefficients  $n_{\varphi_i} $ give an algebraic count of 
orbits of type $(H^{\varphi_i})$ and as such, completely describe the $\Gamma \times S^1$-equivariant $J$-homomorphism of an operator involved.

In this paper, we do not compute this total invariant. Instead, we just compute {the $S^1$-equivariant twisted degree 
(in short, $S^1$-degree)} of the {associate} operator restricted to $H^\varphi$-fixed point space {(essentially, we show the non-triviality of the corresponding $S^1$-equivariant $J$-homomorphism).}  The advantage of this approach is that in a number of circumstances, which we illustrate by examples, the total $\Gamma \times S^1$-twisted degree is not defined, however we still succeed to detect branches with various symmetric properties.

{The usage of the total $\Gamma \times S^1$-equivariant twisted degree is effective for  studying {\it global} behavior of branches of periodic solutions,} and in the case when it is defined, each coefficient $n_{\varphi_i}$ (see \eqref{eq:-equiv-degree}) can be recovered by considering {the usual crossing numbers related to the restrictions of the operator involved to $H^{\varphi_j}$-fixed point subspaces with $(H^{\varphi_j}) > (H^{\varphi_i})$ and applying the so-called Recurrence Formula (see \cite{AED}, p. 124 and Theorem 4.25).}


\begin{remark}{\rm 
To verify condition \textbf{(P2)}(iii), one has to
compute multiplicities of the roots of $\Lambda_1^{H^\varphi} (\alpha_{\pm}, \tau, \beta)$ (cf. \eqref{eq:det-Lambda}).
To this end, decompose $^1\widetilde V$ into its $\Gamma \times S^1$-isotypical components 
$$ ^1 \widetilde V =  \, ^1 \widetilde V_1 \oplus \cdots \oplus \,^1 \widetilde V_q,$$ 
where each $^1\widetilde V_k$ is modeled on the irreducible $\Gamma \times S^1$-representation $^1\widetilde {\mathcal V}_k$
($k = 1,...,q$). Fix $\alpha$ and assume that $\lambda_o = \tau_o + i \beta_o$ is a root of $\Delta^{H^\varphi}_1 (\alpha, \tau, \beta)$.
Denote by $E(\lambda_o)$ the (generalized) eigenspace of $\lambda_o$ with respect to $\Delta_1 (\alpha, \tau, \beta)$ and let
$\mathfrak m_k(\lambda) := \dim (E(\lambda_o) \cap  \, ^1 \widetilde V_k)/\dim \,^1\widetilde {\mathcal V}_k$ 
stand for the $^1\widetilde {\mathcal V}_k$-isotypical multiplicity
of $\lambda_o$, $k = 1,...,q$ (cf. \cite{AED}). Put $d_k^{H^\varphi}:= \dim \, ^1\mathcal V_k^{H^\varphi}$.
Then, the multiplicity of $\lambda_o$ considered as a root of $\Lambda_1^{H^\varphi} (\alpha, \tau, \beta)$ is given by 
$$
\sum_{k=1}^q d_k^{H^\varphi} \mathfrak m_k(\lambda).
$$
}
\end{remark}

\section{Proof of Theorem \ref{thm:main-theorem}}\label{proof}

\subsection{Operator {reformulation}}
 \label{subsec-operator-reform}
 We consider the space $C^1=C^1(S^1;V)$ equipped with the standard norm 
 $\|u\|_{C^1} := \{\sup |u(x)| + \sup |\dot{u}(x)| \, : \, x \in S^1\}$.
 Recall that $S^1$ is identified with the segment $[0,2\pi]$ and the spaces $C$, $C^1$ are identified with the spaces of $2\pi$-periodic functions.
 
Define the differention operator $L=\frac{d}{dt}: C^1\to C$ and the 
projector 
$K: C^1\to {C}$ onto the subspace of constant functions given by
\[
K u (t) =\frac1{2\pi}\int_0^{2\pi} u(s)\,ds.
\]
We note that the operator $L+K$ maps 
$C^1$  onto ${C}$ and is invertible. Its inverse operator is defined by 
\[
((L+K)^{-1} u)(t)=\int_0^{2\pi} H(t-s) u(s)\,ds,
\]
where 
\[
H(\tau)=\frac1{2\pi}(1+\pi-\tau), \ \ \ 0\le \tau<2\pi; \qquad H(\tau+2\pi)=H(\tau), \ \ \ \tau\in\mathbb{R},
\]
is the impulse response function of the linear periodic problem
\[
\dot v + \frac1{2\pi}\int_0^{2\pi} v(s)\,ds = u, \qquad u(0)=u(2\pi).
\]
In other words, the bounded operator $(L+K)^{-1}: {C}\to C^1$ is the solution operator of this problem, i.e.~$v=(L+K)^{-1}u$.

Rewriting {\eqref{eq:per-normalization}} as an equivalent equation
\[
\dot u + \frac1{2\pi} \int_0^{2\pi}u(s)\,ds=\beta^{-1} A (\alpha) u + \beta^{-1} f(\alpha, u) + \frac1{2\pi} \int_0^{2\pi}u(s)\,ds
\]
with the $2\pi$-periodic boundary conditions, we see that the periodic problem for \eqref{PNH} is equivalent to the fixed point problem
\begin{equation}\label{eq:fixed-point}
u= {\mathcal J} (L+K)^{-1} \bigl(\beta^{-1}A(\alpha)u +Ku +\beta^{-1} F(\alpha, u)\bigr) =: {T(\alpha,\beta,u)}
\end{equation}	
in the space $\mathbb R^2 \oplus {C}$, where $\mathcal{J}$ is the compact embedding operator from $ C^1$ to 
$ {C}$ and $F : \mathbb R \oplus C \to C$ is given by $F(\alpha,u)(t):= f(\alpha,u(t))$. 
{Also, by condition {\bf (P0)} and compactness of $\mathcal{J}$, the  vector field $\Id - T$ 
is compact.  In addition, formula 
$$(g,e^{i \theta }) u(t):= g u(t - \theta), \quad\quad (g,e^{i\theta}) \in \Gamma \times S^1 = G,$$ 
defines  {\it isometric} Banach $G$-representations on ${C}$ and $ C^1$ and  $\Id - T: 
\mathbb R^2 \oplus C \to C$ is $G$-equivariant (we assume that $G$ acts trivially on $\mathbb R^2$). In what follows, for any $s \in (r,R)$, we are going 
to prove the existence of a solution $(\alpha,\beta,u)$ to \eqref{eq:fixed-point} such that $\|u\|_C = s$ and $G_u = H^{\varphi}$. Due to $G$-equivariance, this is equivalent to studying the solution set of the equation 
\begin{equation}\label{eq:wth-s}
\mathfrak{F}_s(\alpha,\beta,u):= \Big(\|u\|_C - s, u - T(\alpha,\beta, u)\Big)=0,
\end{equation}   
where $(\alpha,\beta) \in \mathbb R^2,$ $u \in C^{H^{\varphi}},$ $s \in (r,R).$ 
}
	
\subsection{Auxiliary lemmas}
{It is easy to see that the subspace $C^{H^{\varphi}} \subset C$ is an isometric $S^1$-representation. Therefore, 
solutions to  \eqref{eq:wth-s} will be studied in the subset $[\alpha_-,\alpha_+] \times [\beta_-,\beta_+] \times C^{H^{\varphi}} \subset \mathbb R^2 \oplus C$ using the $S^1$-degree theory (see  Subsection \ref{subsec:S1-degree}).  As it is common for the application of any (equivariant) degree based methods, our approach includes the following steps:} 

{(a) Construction of an open bounded $S^1$-invariant domain $\Omega \subset [\alpha_-,\alpha_+] \times [\beta_-,\beta_+] \times C^{H^{\varphi}}$ such that related fields and homotopies are $\Omega$-admissible;}

{(b) Construction of an $\Omega$-admissible $S^1$-equivariant deformation of $\mathfrak{F}_s$ to an associated linear field $a_s$;}

{(c) Showing that the $S^1$-degree of  $a_s$ is different from zero;}

{(d) Establishing the existence of (connected)  branches of solutions for equation \eqref{eq:wth-s}.}

\medskip
{To simplify our notations, we identify the set 
$\mathcal D \subset [\alpha_-,\alpha_+] \times \{0\}  \times \mathbb R_+$ with the subset  
of $ [\alpha_-,\alpha_+] \times \mathbb R_+$ via $(\alpha,0,\beta) \to (\alpha,\beta)$ (cf. condition {\bf (P3)}) for which we use the same symbol 
$\mathcal D$. Also, put  
$\mathfrak W := C^{H^\varphi}$ and observe that $\mathfrak W$ 
admits the $S^1$-isotypical decomposition
\begin{equation}\label{eq:isotypical-l}
\mathfrak W = \overline{\bigoplus_{l=0}^{\infty} \mathfrak W_l},
\end{equation}
where $\mathfrak W_0$ coincides with $V^H$-valued constant functions and as such, can be identified with the subspace $V^H$ of the phase space $V$, while $\mathfrak W_l$ can be identified with $^l\til V^{H^\varphi}$(see Section \ref{subsec:main-results}). 

\begin{remark}\label{rem:invertibility}
{\rm Due to conditions {\bf (P1)} and {\bf (P3)}(iii), the operator $\beta L - A(\alpha) : (C^1)^{H^{\varphi}} \to \mathfrak W$ is invertible for every 
$(\alpha,\beta) \in \partial \mathcal D$. Therefore, applying the same argument as in Subsection \ref{subsec-operator-reform}, one can easily show that equation \eqref{eq:wth-s} restricted to $\partial \mathcal D \times \mathfrak M$ is equivalent to
\begin{equation}
\Big(\|u\|_C - s, u - \mathcal J(\beta L - A)^{-1}F(\alpha,u)\Big) = 0,
\end{equation}
where
$(\alpha,\beta) \in \partial \mathcal D,$ $u \in \mathfrak W,$ $s \in (r,R).$}
\end{remark}
\noindent
Define $\Omega \subset [\alpha_-,\alpha_+] \times [\beta_-,\beta_+] \times \mathfrak W$ by 
\begin{equation}\label{eq:domain-Omega}
\Omega: =  \mathcal D \times B_R(0),           
\end{equation}
where $B_R(0) := \{u \in \mathfrak W \, : \, \|u\|_C < R\}$ (cf. conditions {\bf (P3)}--{\bf (P5)}). The following statement is crucial for our considerations. 

\begin{lemma}\label{l1}
Assume that  conditions {\bf (P0)}, {\bf (P1)},  {\bf (P3)}--{\bf (P5)} are satisfied and $\Omega$ is defined by \eqref{eq:domain-Omega}. Then, for any  $\mu \in [0,1]$ and any $s \in (r,R)$, the equation 
\begin{equation}\label{eq:parameter}
\mathcal{F}_s(\alpha,\beta,\mu,u) := \Big(\|u\|_C - s, u - \mu\mathcal J(\beta L - A)^{-1}F(\alpha,u)\Big) = 0
\end{equation}
does not have solutions on $\partial \Omega$.
\end{lemma}
 }

\begin{proof} {Due to the restrictions on $s$, equation \eqref{eq:parameter} does not admit solutions on 
\begin{equation}\label{eq:dom1}
\partial \mathcal D \times \{u \in \overline{B_0(R)} \, : \, \|u\|_C = R \,\, {\rm or} \,\, 0 \leq \|u\|_C \leq r \}.
\end{equation}
For contradiction to the statement of the lemma, assume that \eqref{eq:parameter} admits a solution on
\begin{equation}\label{eq:dom2}
\partial \mathcal D \times \{u \in \overline{B_R(0)} \, : \, r  < \|u\|_C < R \}.
\end{equation}
Denote $v :=F(\alpha,u)$. With this notation, \eqref{eq:parameter}
implies
\[
u=\mu \mathcal J (\beta L - A(\alpha))^{-1} v.
\]
Hence, }
\begin{equation}\label{prohh}
\|u\|_C \le  \mu 
\|(\beta L - A(\alpha))^{-1}\|_{L^2\to C} \|v\|_{L_2} \leq  \|(\beta L - A(\alpha))^{-1}\|_{L^2\to C} \|v\|_{L_2}.
\end{equation}
On the other hand, according to {\bf (P4)}, the relations $ r {<} \|u\|_C {<} R$ imply $\|v\|_C=\|F(\alpha,{u})\|_C\le N(\alpha)\|u\|_C$. Combining this estimate with \eqref{prohh} and $\|v\|_{L_2}\le \sqrt{2\pi}\|v\|_C$, we obtain
\begin{equation}\label{eq:estim1}
\|u\|_C \le  \sqrt{2\pi} N(\alpha)\|(\beta L - A(\alpha))^{-1}\|_{L^2\to C}\|u\|_C.
\end{equation}
The quantity $\|(\beta L - A(\alpha))^{-1} \|_{L_2\to C}$ has already been defined in 
\eqref{linear_bound} as $M^{H^\varphi}(\alpha,\beta)$. By 
{\bf(P5)},
\begin{equation}\label{eq:estim2}
q: = \sqrt{2\pi} N(\alpha) M^{H^\varphi}(\alpha,\beta) < 1.
\end{equation}
This {together with \eqref{eq:estim1}} gives
$$
\|u\|_C\le q\|u\|_C<\|u\|_C,
$$ 
which is a contradiction.
\end{proof}

Define the vector field
\begin{equation}\label{10*}
a_{s,\mu}(\alpha,\beta,u)=\Big(\|u\|_C-s, u- J(L+K)^{-1}\bigl(\beta^{-1}A(\alpha)u +Ku +\mu\beta^{-1} F(\alpha, u)\Big)
\end{equation}
for $(\alpha,\beta,u)\in \partial \Omega$. We note that each of the vector fields \eqref{eq:parameter} and \eqref{10*} is equivalent to the periodic problem
\[
\begin{cases}
\beta \dot u= A(\alpha)u+\mu f(\alpha,u),\\
u(0) = u(2\pi), \ \|u\|_C=s.
\end{cases}
\]
{Therefore, as a consequence of Lemma \ref{l1}, one has the following statement.

\begin{corollary}\label{cor-homotopy}
Assume that conditions {\bf (P0)}, {\bf (P1)},  {\bf (P3)}--{\bf (P5)} are satisfied and $\Omega$ is given by \eqref{eq:domain-Omega}. Then, for any 
$s \in (r,R)$, the vector field 
$\mathfrak F_s$ given by \eqref{eq:wth-s} is $\Omega$-admissibly homotopic to the vector field  
\begin{equation}\label{eq:linear}
a_s(\alpha,\beta, u) := \Big(\|u\|
_C - s, u - {\mathcal J} (L+K)^{-1} \bigl(\beta^{-1}A(\alpha)u +Ku)\Big). 
\end{equation} 
In particular, $S^1\text{\rm -deg}\,(\mathfrak F_s, \Omega)$ and $S^1\text{\rm -deg}\,(a_s, \Omega)$ are correctly defined and coincide (see Subsection \ref{subsec:S1-degree}, property {\bf (A2)}).
\end{corollary}
}

\subsection{Computation of $S^1\text{\rm -deg}\,(a_s, \Omega)$}
{Corollary \ref{cor-homotopy} essentially reduces studying the solution set of equation \eqref{eq:wth-s} to the computation of $S^1\text{\rm -deg}\,(a_s, \Omega)$. To this end, it is convenient to identify $\overline{\mathcal D}$ with a subset of $\mathbb C$ via 
$(\alpha,\beta) \to \lambda = \alpha + i\beta$, and using {\bf (P3)}(i), to assume without loss of generality that $\overline{\mathcal D}$ is a closed disc of radius $\varepsilon$
centered at $\lambda_o$. 
Put
\begin{equation}\label{eq:a-lambda}
a(\lambda)u := u - {\mathcal J} (L+K)^{-1} \bigl(\beta^{-1}A(\alpha)u +Ku),
\end{equation}
and denote by  $a_l(\lambda)$ the restriction of $a(\lambda)$ to $\mathfrak W_l$ (see \eqref{eq:isotypical-l}). Also,
put
\begin{equation}\label{def_n_k}
\mathfrak n_0 := {\rm sign}\,(\text{det}(a_0(\lambda))), \quad \mathfrak n_l := {\rm deg}\,(\text{det}_{\mathbb C}(a_l(\cdot)),\mathcal D),
\end{equation} 
where ``${\rm deg}$" stands for the usual winding number. By condition {\bf (P1)} (resp. {\bf (P3)}(iii)), $\mathfrak n_0$ is independent 
of $\lambda \in \overline{\mathcal D}$ (resp. $\mathfrak n_l$ is correctly defined).     
Observe also that by compactness of the vector field \eqref{eq:a-lambda}, only finitely many $\mathfrak n_l$ are different from zero.
}

\begin{lemma}\label{lem_deg_val} {Under the assumptions {\bf (P0)}, {\bf (P1)},  {\bf (P3)}--{\bf (P5)} and notations 
\eqref{eq:domain-Omega}, \eqref{eq:linear} and \eqref{def_n_k}, one has}  
\begin{equation}\label{eq:formula-S^1-lin}
S^1\text{\rm -deg}\,(a_s, \Omega) = \mathfrak n_0 \sum_{l=1}^{\infty} \mathfrak n_l(\mathbb Z_l)
\end{equation}
{for every $ r < s < R$.}
\end{lemma}

\begin{proof} 
	{We will use a modification of the argument given in \cite{AED}.}
The main strategy is to deform the vector field $a_s$ and to modify 
$\Omega$
in such a way that the computational formula \eqref{eq:main-formula} combined with property {\bf (A5)}  of the degree
(see Subsection \ref{subsec:S1-degree}) can be applied. 
{The proof follows three main steps. First, we make a finite-dimensional approximation of the compact vector field $a_s$. We note that each subspace $\mathfrak W_l$ of $\mathfrak W$ is invariant for the compact linear map $a(\lambda)$, and so is any subspace $\mathfrak W^m:= {\bigoplus_{l=0}^m \mathfrak W_l}$. We choose a sufficiently large subspace 
$\mathfrak W^m$  and fix a (closed) linear subspace $Y \subset \mathfrak W$ complementing $\mathfrak W^m$ (without loss of generality, one can assume that $Y$ is also invariant for $a(\lambda)$). 
Now, we define
\begin{equation}\label{eq:decompos}
a^m(\lambda) := a(\lambda)|_{\mathfrak W^m} + \Id |_Y, \quad\quad a^m_s(\lambda,u) := (\|u\|_C - s, a^m(\lambda)u).
\end{equation}
Due to compactness of $a_s$,
the linear homotopy joining $a_s$ and   $a^m_s$ is   $\Omega$-admissible for a sufficiently large $m$.  Put
\begin{equation}\label{eq:finite-dimen}
\Omega^m := \Omega \cap (\mathbb C \oplus \mathfrak W^m), \quad\quad \tilde{a}^m_s := a^m_s |_{\overline{\Omega^m}}.
\end{equation}
Using properties {\bf(A2)} and {\bf(A6)}  (see Subsection \ref{subsec:S1-degree}), one obtains:
\begin{equation}\label{eq:deg-approx}
S^1\text{\rm -deg}\,(a_s, \Omega) = S^1\text{\rm -deg}\,(a^m_s, \Omega) = S^1\text{\rm -deg}\,(\tilde{a}_s^m, \Omega^m).
\end{equation} 
Let $P_l : \mathfrak W^m \to  \mathfrak W_l$ be a canonical equivariant projection (see, for example, \cite{AED}, p. 36).
Then, $\tilde{a}_s^m$ is given by
\begin{equation}\label{eq:projection-l}
\tilde{a}_s^m(\lambda,u) = \Big(\|u\|_C - s, \bigoplus_{l = 0}^ma_l(\lambda)P_lu\Big),
\end{equation}  
Since $\overline{\mathcal D}$ is contractible to $\lambda_o$, there exists a deformation $\mu : \overline{\mathcal D} \times [0,1] \to \overline{\mathcal D}$ such that $\mu(\lambda,0) = 
\lambda$ and $\mu(\lambda,1) \equiv \lambda_o$. 
Since $a_0(\lambda)$ is invertible for every  $\lambda \in \overline{\mathcal D}$, formula
\begin{equation}\label{eq:const-for-a0}
\hat{a}(\lambda,u,\nu):= \Big(\|u\|_C - s, a_0(\mu(\lambda,\nu))P_0 u + \bigoplus_{l=1}^m a_l(\lambda) P_l u\Big) 
\end{equation}
determines an $\Omega^m$-admissible homotopy joining $\tilde{a}_s^m$ with the vector field $\hat{a}$ defined by
\begin{equation}\label{eq:after-def}
\hat{a}(\lambda,u) := \Big(\|u\|_C - s, a_0(\lambda_o)P_0 u + \bigoplus_{l=1}^m a_l(\lambda) P_l u\Big).
\end{equation}
Put
\begin{equation}\label{eq:defs}
B_0: = \{u \in  \mathfrak W_0 \,:\, \|u\|_C < R\}, \;\;\;\; \Omega_{\ast} := \Omega^m \cap \Big(\mathbb C \oplus \bigoplus_{l=1}^m W_l \Big), \;\;\;\;  a_{\ast} := \hat{a}|_{\Omega_{\ast}}.
\end{equation}
Since 
\begin{equation}\label{eq:product-map}
\hat{a} = a(\lambda_o) \times a_{\ast} : \overline{B_{0} \times \Omega_{\ast} }\to  \mathfrak W_0 \times \Big(\mathbb R \oplus \bigoplus_{l=1}^m  
\mathfrak W_l\Big),
\end{equation}
one has (thanks to property  {\bf (A5)} of the degree,  Subsection \ref{subsec:S1-degree}):
\begin{equation}\label{eq:product-map1}
S^1\text{\rm -deg}\,(a_s, \Omega) = n_0 \cdot S^1\text{\rm -deg}\,(a_{\ast}, \Omega_{\ast}).
\end{equation} 
Finally, to compute $S^1\text{\rm -deg}\,(a_{\ast}, \Omega_{\ast})$, take $\xi :   \Omega_{\ast} \to  \mathbb R \oplus \bigoplus_{l=1}^m  \mathfrak W_l$ defined by
$\xi(\lambda,u):= |\lambda - \lambda_o|(\|u\|_C - r) + \|u\|_C + \varepsilon {r / 2}$, put  $\overline{a}:=  \bigoplus_{l=1}^m a_l(\lambda) P_l u$ and observe that the field $a_{\xi} := (\xi, \overline{a})$ is $\Omega_{\ast}$-admissibly homotopic to  $a_{\ast}$.  Since  
\begin{equation}\label{eq:zeros1}
a_{\xi}^{-1}(0) = \{(u,\lambda) \in  \Omega_{\ast} \, : \, u = 0, \; \; \;\;  |\lambda - \lambda_0| = {\varepsilon/ 2}\},
\end{equation}
one can combine property {\bf (A2)} of the degree with its excision property (cf. Remark \ref{rem:excision}) to obtain: 
\begin{equation}\label{eq:applic-excision}
 S^1\text{\rm -deg}\,(a_{\ast}, \Omega_{\ast}) = S^1\text{\rm -deg}\,(a_{\xi},\Omega_{\ast} ) = S^1\text{\rm -deg}\,(a_{\xi},\Omega_1),
\end{equation}
where 
$$
\Omega_1 := \{(u,\lambda) \in  \Omega_{\ast} \, : \, {\varepsilon / 4} < |\lambda - \lambda_o| < \varepsilon \}.
$$
For any $(\lambda,u) \in \Omega_1$, set
$$
\eta(\lambda):= \lambda_o + {\varepsilon (\lambda - \lambda_o) \over 2|\lambda - \lambda_o)|}
$$
and define $b : \Omega_1 \to \mathbb R \oplus  \mathfrak W^m$ by $b(\lambda,u) := (\xi,\overline{a}(\eta(\lambda))u)$. 
From \eqref{eq:zeros1} it follows that  
$a_{\xi}^{-1}(0)  = b^{-1}(0)$, hence $a_{\xi}$ and 
$b$ are $\Omega_1$-admissibly homotopic. 
To complete the proof, it remans: to take a homemorphism of $\Omega_1$ onto  $\mathcal O$ (see \eqref{eq:mathcal-O}), replace the above function 
$\xi$ by $|\lambda|(\|v\|_C-1)+\|v\|_C+1$ (see \eqref{eq:f-a}) 
and apply formula \eqref{eq:main-formula}.}
\end{proof}

{As a consequence of Lemma \ref{lem_deg_val}, we have the following statement.}
	
\begin{corollary}\label{cor:S1-non-zero}
Under the assumptions {\bf (P0)}--{\bf (P5)} and notations \eqref{eq:domain-Omega}, \eqref{eq:linear} and \eqref{def_n_k}, 
$$S^1\text{\rm -deg}\,(a_s, \Omega)  \not=0.$$ 
\end{corollary}

\begin{proof}
By condition {\bf (P2)}(i), the local Brouwer degree of $\Lambda^{H^{\varphi}}_1 :  \partial \mathcal P \to \mathbb C$ is correctly defined and equal to zero.
Combining this with condition {\bf (P2)}(ii) and excision of the local Brouwer degree yields
\begin{equation}\label{eq:excis}
\deg(\Lambda^{H^\varphi}_1,\partial \mathcal P) = \deg(\Lambda^{H^\varphi}_1,\mathcal P_+)  + \deg(\Lambda^{H^\varphi}_1,\mathcal P_-) + 
\deg(\Lambda^{H^\varphi}_1,\mathcal P_0) = 0.
\end{equation} 
Denote by $\mathfrak t_{\pm}$ the number of roots of  $\Lambda^{H^\varphi}_1$ in $\mathcal P_{\pm}$ (counted according to their multiplicities). Obviously, $\mathfrak t_{\pm} = \pm \deg(\Lambda^{H^\varphi}_1,\mathcal P_{\pm})$. Combining this with the $\mathbb Z_2$-equivariance of
$\Lambda^{H^\varphi}_1$, conditions {\bf (P2)}(iii) and {\bf (P3)}(ii) and formula \eqref{eq:excis} yields
$$
0 \not= \mathfrak t_- - \mathfrak t_+ = \deg(\Lambda^{H^\varphi}_1,\mathcal P_0) = 2 \deg(\Lambda^{H^\varphi}_1,\mathcal D) = 2\mathfrak n_1,
$$
and the result follows from \eqref{eq:formula-S^1-lin}.
\end{proof}

\subsection{Completion of the proof of Theorem \ref{thm:main-theorem}}
{Combining  Corollaries \ref{cor-homotopy} and \ref{cor:S1-non-zero} with properties {\bf (A1)} and {\bf (A2)} of the degree (see Subsection 
\ref{subsec:S1-degree}) implies the existence of a solution to equation \eqref{eq:wth-s} for each $s\in(0,1)$ and then by compactness also for $s=0,1$. }  

{To show that these solutions are not constant, notice that if $(\alpha_1,\beta_1, u_1)$ is a constant solution of \eqref{eq:wth-s} then $
(\alpha_1,\beta_{\textcolor{blue}{\ast}}, u_1)$ is a solution of \eqref{eq:wth-s} for any $\beta^*$. In particular, we can choose $\beta^*$ such that $
(\alpha_1,\beta^*)\in\partial\mathcal D$ which contradicts Lemma \ref{l1}. 
}

{Finally, to show that the solution set to equation \eqref{eq:fixed-point} contains a compact connected branch joining the spheres 
$\{\|u\|_C=r\}$ and  $\{\|u\|_C=R\}$, one can use the standard technique ((see, for example, \cite{AED,orig_slid} for details) based on the statement 
following below (see \cite{Kurat}, Theorem 3, p. 170).

\begin{lemma}[Kuratowski]\label{lem:Kur}
Let $X$ be a metric space, $A,B \subset X$ two disjoint closed sets, and $K$ a compact set in $X$ such that $K \cap A \not=\emptyset\not=K \cap B$.
If the set $K$ does not contain a connected component $K_o$ such that $K_o \cap A \not=\emptyset\not=K_o \cap B$, then there exist two disjoint open sets $V_1$ and $V_2$ such that $A \subset V_1$, $B \subset V_2$ and $A \cup B \cup K \subset V_1 \cup V_2$.
\end{lemma}
}

\section{Examples}\label{examples}
In this section we 
consider applications of Theorem \ref{thm:main-theorem}. 
In the first example we consider a system without symmetry and compute estimates of $r$ and $R$.
 That is, we guarantee not only the existence of a branch of periodic solutions but also estimate its length.
Further examples refer to a number of 
circumstances where 
standard genericity assumptions are not satisfied but Theorem \ref{thm:main-theorem} is applicable. 
Here estimates of the length of the branch
are also possible to obtain but for convenience we consider nonlinearities with  sublinear growth and only present results of the form ``there exists an $R$ such that there is a branch of solutions joining the trivial equilibrium to the sphere of radius $R$''.

In the second example, we consider a system of coupled oscillators which undergoes a Hopf bifurcation and a steady state bifurcation simultaneously. In this case, restriction to $H^\varphi$-fixed point spaces allows us to ``separate'' these two bifurcations. An additional assumption that the nonlinearity of individual oscillators is odd allows us to refine these results to be more inclusive of various {non-generic} scenarios.

Finally, in the third example we treat the case when a pair of purely imaginary eigenvalues persists
independently of the bifurcation parameter, and other eigenvalues cross through them as the parameter is varied. Again, in this case the restriction to $H^\varphi$-fixed point spaces allows us to ``separate'' these eigenvalues.

\subsection{Example 1}
Consider a single Van der Pol oscillator given by 
\begin{equation}\label{ind_vdp}
\begin{aligned}
\dot x_1 &= x_2,\\
\dot x_2 &= -x_1 - x^2_1 x_2 + \alpha x_2.
\end{aligned}
\end{equation}
In this case, the group $H^\varphi$ is trivial, $$A(\alpha)= \begin{bmatrix}
  0 & 1\\
  -1 & \alpha 
\end{bmatrix}$$ and $f(\alpha, x) = (0,-x_1^2x_2)^T.$ To apply the main theorem we begin by identifying that $$\Lambda_l(\alpha, \tau, \beta)= l^2(\tau+i\beta)^2 - \alpha l(\tau+i\beta) + 1.$$ Given any $\alpha_-<0<\alpha_+$, it is easy to see that if we take $$\mathcal P = [\alpha_-,\alpha_+] \times [0,\tau_*] \times [0,\beta_*]$$
with sufficiently large $\tau_*, \beta_*$, then
conditions {\bf(P0)-(P2)} are satisfied. The main challenge is now to construct the set $\mathcal D$ in such a way that condition {\bf (P3)} is satisfied and the estimate for $R$ can be optimized. Satisfying condition {\bf (P3)} only requires  that $(\alpha,\beta)=(0,1)\in \mathcal D$, while $(0,1/l) \notin \partial \mathcal D$ for $l\ge 2$. In particular, if $(\alpha,\beta)\neq (0, 1/l)$ for $l\in\mathbb{N}$, then the number  \eqref{linear_bound} is given by
\[
M(\alpha,\beta)=\Big( \sum_{l=0}^\infty \frac{2(l^2\beta^2+1)+\alpha^2}{(l^2\beta^2-1)^2+l^2\alpha^2\beta^2} \Big)^{1/2}.
\]
Hence, if we take $\mathcal D$ to be any  sufficiently  small disk surrounding the point $(0,1)$, then $$
\mathcal N := \inf_{(\alpha,\beta)\in \partial \mathcal D} \frac{1}{\sqrt{2\pi}M(\alpha,\beta)} >0.
$$
The
form of $f$ implies that $|f(\alpha,x)|\le \mathcal N|x|$ for all $x=(x_1,x_2)^T$ with $|x|\le \sqrt \mathcal N$. We can therefore see that conditions {\bf (P4)}, {\bf (P5)} are satisfied with $N(\alpha)=\mathcal N$, $r=0$ and any $R <\sqrt \mathcal N$. Therefore Theorem \ref{t1} guarantees the existence of a branch of periodic solutions joining the zero equilibrium with the sphere $\{\|x\|_C=\sqrt{\mathcal N}\}$.

However we can try to maximize the number $\mathcal N$ by choosing an appropriate domain $\mathcal D$. For $\mathcal D$ to include the point $(0,1)$ the boundary $\partial \mathcal D$ must intersect the line segment $(0, \beta)$, $\beta \in (0,1)$. If we therefore identify the point $(0, \beta_*)$ which minimizes the function 
$
M(0,\beta)
$ 
along this line segment and take $\partial {\mathcal D}$ to be the level curve of $M(\alpha,\beta)$ passing through that point, then we can maximize the number $\mathcal N$. 
Noting that the function
\[
M(0,\beta)=1+\left(\frac{\pi}{\beta} \csc \frac{\pi}{\beta} \right)^2
\]
achieves its global minimum at the point $\beta_*\approx0.699$ which is determined as a root of the equation $\tan (\pi/\beta)=\pi/\beta$,
in this way we obtain 
that there exists a branch of periodic solutions joining the trivial equilibrium to the sphere of radius $R\approx 0.3$, see Figure \ref{f1}. The same scheme can be used to construct the domain $\mathcal D$ satisfying the conditions of Theorem \ref{t1} and obtain an estimate for $R$
for the general non-equivariant system \eqref{PNH} undergoing a generic Hopf bifurcation.

\begin{figure}[htb!]
	\centering
		\includegraphics[width=0.45\textwidth]{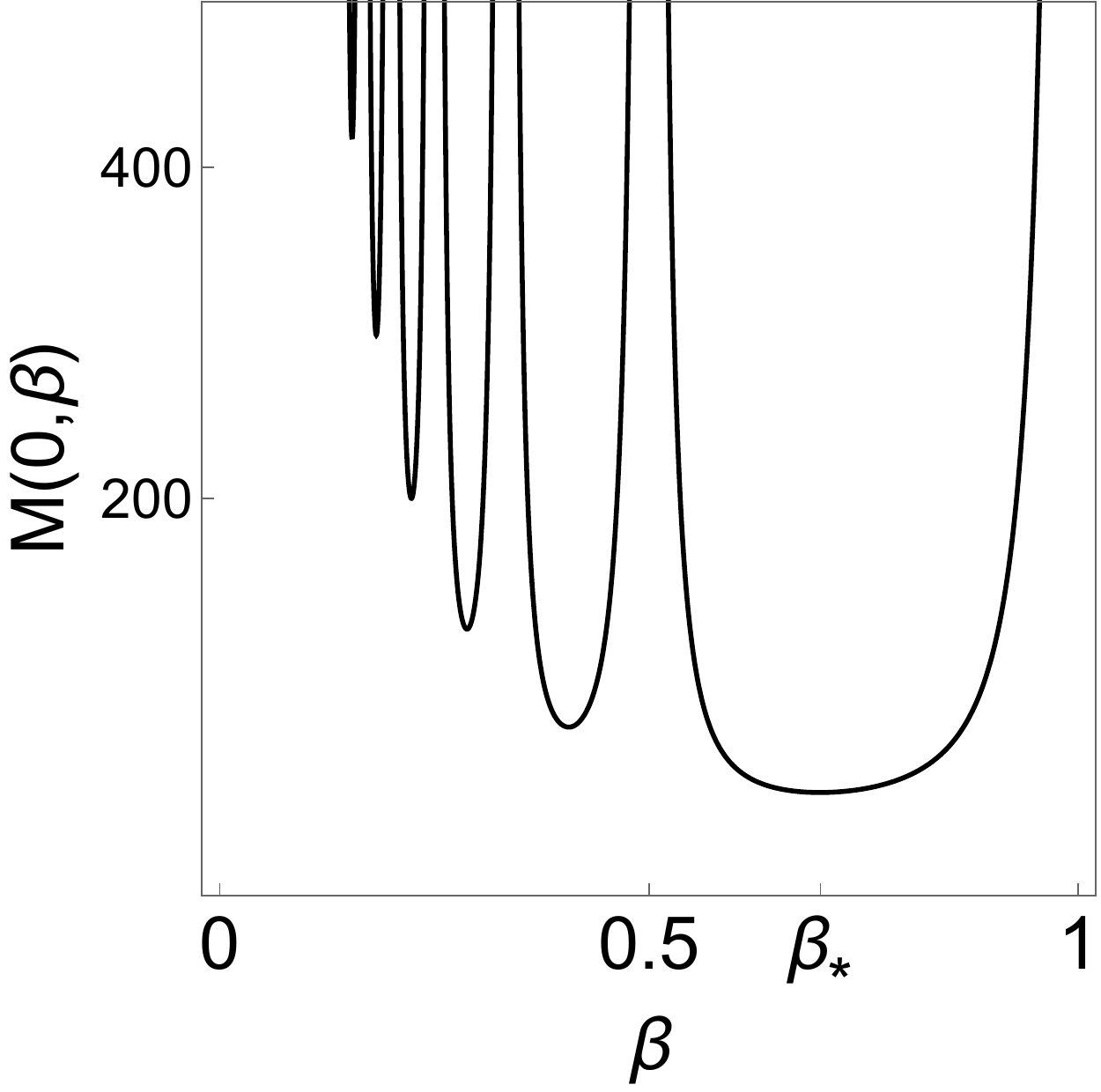} \qquad		\includegraphics[width=0.46\textwidth]{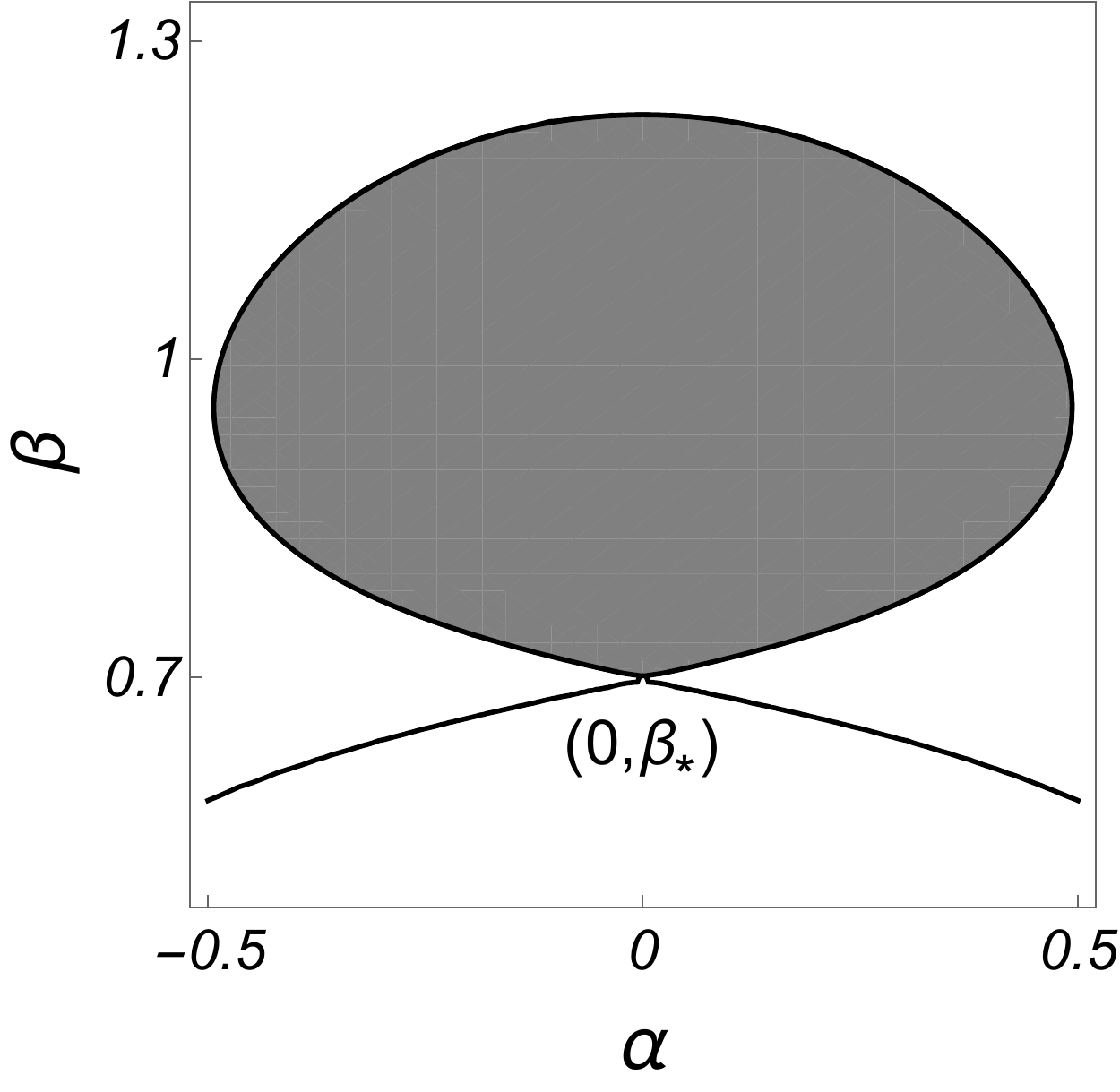}\\
\qquad	{\bf (a)} \qquad\qquad\qquad\qquad\qquad\qquad\qquad\qquad {\bf (b)} 
		\caption{Left: The function $M(0,\beta)$ with the global minimum at the point $\beta_*\approx 0.7$. Right: The domain ${\mathcal D}$ (gray) bounded by the level curve $M(\alpha,\beta)=M(0,\beta_*)$ of the function $M(\alpha,\beta)$.}
\label{f1}
\end{figure}

\subsection{Example 2}
Our second example illustrates the situation where the spectrum of $A(\alpha)$ contains purely imaginary eigenvalues for some value $\alpha_0$ of the parameter $\alpha$ and, simultaneously, $\det A(\alpha_0) = 0$. This contradicts the usual ``absence of the steady state bifurcation'' condition. However, we overcome this by considering \eqref{eq:det-Lambda} instead of \eqref{eq:without} with properly chosen $H^{\varphi}$ and applying 
Theorem \ref{thm:main-theorem}. In the following example complex eigenvalues cross the imaginary axis transversally at the bifurcation point $\alpha=\alpha_0$. We use the notations adopted 
in \cite{BKRZ-Hysteresis} and \cite{HBKR}.

We begin with the system
\begin{equation}\label{eq:VdP-1}
\begin{aligned}
\dot j_m &= -{R \over L} j_m + {1 \over L} u_m,\\
\dot u_m &= -{1 \over C} j_m +{\alpha \over C} u_m - {\sigma \over C} u^3 _m
\end{aligned}
\end{equation}
describing an LCR circuit with a cubic current-voltage characteristic (here $\alpha$ is a bifurcation parameter),
which can be rewritten as the classical Van der Pol equation.  We further consider a symmetrically coupled system of eight identical oscillators  \eqref{eq:VdP-1}, which are arranged in a cube-like configuration. More precisely, using the vector notation $j= (j_1,...,j_8)^T$, $u = (u_1,...,u_8)^T$ and $u^3 = (u_1^3,...,u_8^3)^T$, one can represent the corresponding system as follows:
\begin{equation}\label{eq:VdP-8}
\begin{aligned}
\dot j &= -{R \over L} j + {1 \over L} u,\\
\dot u &= -{1 \over C} j +{\alpha \over C} u - {\sigma \over C} u^3 + {\rho \over 2C}\mathcal K u,
\end{aligned}
\end{equation}
where   
\begin{equation}\label{eq:K}
\mathcal K =\begin{bmatrix}
    -3 & 1 & 0 & 1 & 1 & 0 & 0 & 0 \\
    1 & -3 & 1 & 0 & 0 & 1 &  0 & 0 \\
    0 & 1 & -3 & 1 & 0 & 0 &  1 & 0  \\
    1 & 0 & 1 & -3 & 0 & 0 & 0&  1  \\
     1 & 0 & 0 & 0 & -3 & 1 & 0 & 1 \\
     0 & 1 &  0 & 0 & 1 & -3 & 1 & 0 \\
     0 & 0 &  1 & 0 & 0 & 1 & -3 & 1   \\
     0 & 0 & 0&  1 & 1 & 0 & 1 & -3 
    \end{bmatrix}
\end{equation}
(we assume that the oscillators are coupled by resistors having the same   conductivity $\rho$ as in \cite{BKRZ-Hysteresis}).  Denote 
by $V:= \mathbb R^{16}$ the phase space of \eqref{eq:VdP-8}. Clearly, $V$ is an $O_4$-representation, where $O_4 =  S_4 \times O_1$ 
acts by permuting pairs of coordinates $(j_m,u_m)$, $m = 1,...,8$. In addition, system \eqref{eq:VdP-1} respects the antipodal symmetry, meaning that 
system \eqref{eq:VdP-8} is $\Gamma := \mathbb Z_2 \times O_4$-equivariant (see Appendix for the explicit description of $O_1$ and $\Gamma$). The 
$\Gamma$-representation $V$ admits the isotypical decomposition
\begin{equation}\label{eq:isotyp}
V = V_0 \oplus V_1 \oplus V_2 \oplus V_3,
\end{equation} 
where each $V_k$, $k = 0,1,2,3$, is of isotypical multiplicity two and is modeled on the irreducible $\Gamma$-representations $\mathcal W_0,\mathcal W_1, \mathcal W_2, \mathcal W_3$ respectively, which can be described as follows. Let $\mathcal B^+$ and $\mathcal B^-$ be the one-dimensional  
$\mathbb Z_2 \times O_1$-representation where $\mathbb Z_2$ acts antipodally on both $\mathcal B^+$ and $\mathcal B^-$ while $O_1$ acts trivially on $\mathcal B^+$ and antipodally on $\mathcal B^-$.  Let $\mathcal V_0$ (resp. $\mathcal V_ 3$) be the one-dimensional trivial (resp. sign) 
$S_4$-representation, let $\mathcal V_2$ be the natural three-dimensional $S_4$-representation, where $S_4$ acts as a subgroup of $SO(3)$, and let $\mathcal V_1 := \mathcal V_2 \otimes \mathcal V_3$. Then, 
\begin{equation}\label{eq:irred-model}
\mathcal W_0 = \mathcal V_0 \otimes \mathcal B^+, \;\; \mathcal W_1 = \mathcal V_1 \otimes \mathcal B^-, \;\; \mathcal W_2 = \mathcal V_2 \otimes \mathcal B^+, \;\; \mathcal W_3 = \mathcal V_3 \otimes \mathcal B^-.
\end{equation}

Denote by $A=A(\alpha)$ the linearization of the right-hand side of \eqref{eq:VdP-8} at the origin:
\begin{equation}\label{111}
A=
\left[\begin{array}{cc} -{R \over L} & {1 \over L} \\
-{1 \over C} & \alpha \over C
\end{array} \right]\otimes \text{Id}_8 + \rho \left[\begin{array}{cc} 0 & 0\\ 0 & {1 \over 2C}
\end{array}\right] \otimes \mathcal K.
\end{equation}
By choosing an appropriate basis in $V$ respecting isotypical decomposition \eqref{eq:isotyp}, one can show (see \cite{BKRZ-Hysteresis,HBKR}) that
$A(\alpha)$ admits a block diagonal representation with $8$ two-by-two blocks
\begin{equation}\label{eq:2-2-blocks}
 A_k(\alpha)= \begin{bmatrix}
  -\frac{R}{L} & \frac{1}{L}\\
  -\frac{1}{C} & \frac{\alpha}{C} - \frac{k\rho}{C}
  \end{bmatrix},
 \end{equation}
where $A(\alpha)|_{V_k} = A_k(\alpha)$ for $k = 0,3$ and $A(\alpha)|_{V_k} = A_k(\alpha) \oplus A_k (\alpha)\oplus A_k(\alpha)$ for $k = 1,2$.
Further, if ${R^2C} < L$, then $A(\alpha)$ has purely imaginary eigenvalues when $$\alpha = \alpha^h_j:=  {RC}/{L} + {j}\rho,\qquad j=0,1,2,3,$$ and $A(\alpha)$ is not invertible when 
$\alpha = \alpha^s_k ={1}/{R}+ {k} \rho$, $k = 0,1,2,3$. Following \cite{BKRZ-Hysteresis}, assume that ${R^2C} < L$ and define
$$
\mathcal C := \frac{1}{\rho R}\Big(1-\frac{R^2C}{L}\Big) > 0.
$$
With this notation, the scenario when $\alpha^h_j = \alpha^s_k$ corresponds to
\begin{equation}\label{eq:simult-Hopf-steady}
\mathcal C = j -k >0\quad \text{for some} \;\; k,j  = 0,1,2,3.
\end{equation} 
Therefore, if $j=0,1,2,3$ and $\mathcal C$ is {\em not} an integer satisfying $\mathcal C\le j$, then the steady state bifurcation is {\em a priori} excluded at the point $\alpha=\alpha^h_j$ (i.e.~$\alpha^h_j\ne \alpha^s_k$ because \eqref{eq:simult-Hopf-steady} is violated), which implies in a standard way that
$\alpha=\alpha^h_j$ is a Hopf bifurcation point for system \eqref{eq:VdP-8}. Moreover, due to symmetries, the Hopf bifurcation points $\alpha=\alpha^h_1$, $\alpha=\alpha^h_2$ give rise to multiple branches of periodic solutions, which can be distinguished by their maximal symmetry group \cite{HBKR,BKRZ-Hysteresis}.
Table \ref{tab_clas} presents spatio-temporal symmetries of multiple branches bifurcationg from the four bifurcation points in this case.

An interesting case is when \eqref{eq:simult-Hopf-steady} holds for some $j,k$.
This case can be handled by Theorem \ref{thm:main-theorem}. In fact, by direct verification (see Appendix), all
the twisted subgroups $H^{\varphi}$ appearing in Table \ref{tab_clas} have the element $(-1,(),{1 \over 2})$. Therefore, $V^H$ (which can be identified with $H$-fixed constant functions) is trivial, hence condition {\bf (P1)} is trivially satisfied. Since all the other assumptions of Theorem  \ref{thm:main-theorem}
are obviously also satisfied, we conclude that  Table \ref{tab_clas} applies to system \eqref{eq:VdP-8} 
in the case \eqref{eq:simult-Hopf-steady} too.
In particular, if $\mathcal C=1,2,3$, some of the Hopf bifurcations listed in Table \ref{tab_clas} are simulatneous with the steady state bifurcation.

\begin{remark}\label{rem:oddness}
{\rm Let us consider system  \eqref{PNH} with $V=\mathbb{R}^{16}$,
where $A(\alpha)$ is given by \eqref{eq:K}, \eqref{111} and $f : \mathbb R \times V \to V$ is an arbitrary 
$\Gamma := \mathbb Z_2 \times O_4$-equivariant continuous function satisfying $f(x,\alpha)/|x|\to 0$ as $x\to 0$ for all $\alpha$. 
Since the above argument was based on the linearization at zero,  Table \ref{tab_clas} applies to this system for every $\mathcal C>0$. 
}
\end{remark}

\medskip
Now, let us consider system \eqref{PNH} with the linear part defined by \eqref{eq:K}, \eqref{111} assuming that $f$ is $O_4$-equivariant but not necessarily $\mathbb Z_2 \times O_4$-equivariant (as in Remark \ref{rem:oddness}). For example, one can think of a system of eight coupled identical oscillators similar to \eqref{eq:VdP-8}, in which the cubic nonlinearity of an individual oscillator is replaced with a polynomial nonlinearity which is not odd. In this case, the results obtained in \cite{BKRZ-Hysteresis,HBKR} imply that
if $\mathcal C\ne 1,2,3$ and $\mathcal C>0$, then Table  \ref{tab_clas} should be slightly modified. Namely, each twisted subgroup ${}^{+}H^{\varphi}$ (resp. ${}^{-}H^{\varphi}$) is replaced by  ${}^{+}\overline{H}^{\varphi}$ (resp. ${}^{-}\overline{H}^{\varphi}$), see Appendix for the explanation of the notation. The question, whether the same table applies in the cases \eqref{eq:simult-Hopf-steady} when a Hopf bifurcation is simultaneous with a steady state bifurcation (i.e., $\mathcal C = 1,2,3$) is more subtle than for the $\mathbb Z_2 \times O_4$-equivariant system considered above. That is, Theorem \ref{thm:main-theorem} can still be used in some of the cases \eqref{eq:simult-Hopf-steady} but not in all of them.


To be more specific, consider, for example, the branch with symmetry $^+\overline{D}^d_4$ which can potentially bifurcate from the trivial solution at $\alpha^h_2$.
By definition, $^+\overline{D}^d_4$ is a graph of the homomorphism $\varphi : D_4 \times O_1 \to \mathbb Z_2 \subset S^1$. By direct computation
 (cf. \eqref{eq:isotyp} and \eqref{eq:2-2-blocks}),
\begin{equation}\label{eq-V-D_4O_1} 
V^{D_4 \times O_1} = V_0 \quad \text{and}  \quad 
A(\alpha^h_2)|_{V^{D_4 \times O_1}}= A_0(\alpha^h_2) = 
\begin{bmatrix}
  -\frac{R}{L} & \frac{1}{L}\\
  -\frac{1}{C} & \frac{\alpha^h_2}{C}
  \end{bmatrix}.
\end{equation}
Combining \eqref{eq-V-D_4O_1} with \eqref{eq:simult-Hopf-steady} implies that if $\mathcal C = 1$, then $\big(\ker A \big) \cap V^{D_4 \times O_1} = \{0\}$,
and condition  {\bf (P1)} is satisfied. A similar argument shows that {\bf (P1)} is also satisfied for other branches bifurcating from the point $\alpha^h_2$ for $\mathcal C=1$. Hence, Theorem \ref{thm:main-theorem} ensures that $\alpha^h_2$ is a Hopf bifurcation point giving rise to multiple branches of periodic solutions with symmetries $(^+\overline{D}_4^d), (^+\overline{D}_3), (^+\overline{D}_2^d), (^+\overline{\mathbb Z}_4^c), (^+\overline{\mathbb Z} _3^t) $ for the $O_4$-equivariant system \eqref{PNH} with $\mathcal C=1$ and, simultaneously, $\alpha^h_2$ is a steady state bifurcation point. However,  if $\mathcal C=2$, then condition {\bf (P1)} is not satisfied at the point  $\alpha=\alpha^h_2$.

Similarly, Theorem \ref{thm:main-theorem} guarantees the Hopf bifurcation of periodic solutions with symmetry $(^-\overline{S}_4^-)$ at the point  $\alpha=\alpha^h_3$ (with a simultaneous steady state bifurcation) for $\mathcal C=2$ but not for $\mathcal C=1$.


\begin{table}[]
\centering
\caption{Symmetries of periodic solutions bifurcating from four Hopf bifurcation points of $\mathbb{Z}_2\times O_4$-equivariant system \eqref{PNH} with the linear part defined by \eqref{eq:K}, \eqref{111} for any $\mathcal C>0$.}

\medskip

\label{tab_clas}
\begin{tabular}{|c|c|}
\hline Bifurcation point & 
Symmetry group of periodic solutions
\\ \hline 
$\alpha =\alpha_0^h$ & $(^+S_4)$ \\ \hline
$\alpha =\alpha_1^h$ &   $(^-D_4^z),  (^-D_3^z),  (^-D_2^d), (^-\mathbb Z_4^c),  (^-\mathbb Z _3^t)$\\ \hline
$\alpha =\alpha_2^h$ &  $(^+D_4^d), (^+D_3), (^+D_2^d), (^+\mathbb Z_4^c), (^+\mathbb Z _3^t)   $\\ \hline
$\alpha =\alpha_3^h$ & $(^-S_4^-)$  \\ \hline
  \end{tabular}
\end{table}

\begin{remark}\label{rem:hysteres}
{\rm 
	Including a ferromagnetic core in an inductor 
	can cause a hysteretic relationship between the magnetic induction $B$ and the magnetic field $H$. In this case, the instantaneous value of $B$ depends not only on the value of $H$ at the same moment, but also on some previous values of $H$. Hence, the constitutive relationship between $B$ and $H$ is an operator relationship,  which translates into a similar operator relationship between the voltage $v_m$ and the current $i_m$ in an LCR contour with a ferromagnetic-core inductor. The Preisach model is a widely used description of such an operator constitutive relationship, defining the dependence of $B$ on $H$ in ferromagnetic materials (see, for example, \cite{Mayergoyz}). The hysteresis memory is the source of {\it non-smoothness} and the presence of an infinite dimensional phase space without local linear structure. Hence, the application of the classical methods based on the centre manifold reduction to systems with hysteresis meets serious difficulties. However, following the scheme described in \cite{BKRZ-Hysteresis,Rachi} and using Theorem \ref{thm:main-theorem}, one can obtain equivariant bifurcation results for networks of LCR circuits with a hysteretic relationship between $B$ and $H$, which are parallel to the results discussed above.
}
\end{remark}

\subsection{Example 3}
In this example, we consider system $\dot x=A(\rho) x + f(\rho,x)$ with the linearization $A=A(\rho)$ given by 
\eqref{111} but this time we use $\rho$ (the coupling strength) as the bifurcation parameter.
The other parameters are fixed. In particular, we assume that $\alpha=RC/L$.
In this case, the spetrum of $A(\rho)$ consists of the eigenvalues $\pm i \omega$ of multiplicity 8 for $\rho=0$ (with $\omega=1/\sqrt{LC}$).
%
%
Let the phase space $V:= \mathbb R^{16}$ of the system be the $O_4$-representation described in Example 2,
and assume again that $f : \mathbb R \times V \to V$ is an $O_4$-equivariant continuous function satisfying $f(\rho,x)/|x|\to 0$ as $x\to 0$.
Formula \eqref{eq:K} implies 
\begin{equation}\label{eq:K-kernel}
\ker \mathcal K = \{(x_1,...,x_8) \, : \, x_1 = \cdots = x_8\} =  (\mathbb R^8)^{O_4},
\end{equation}
therefore $A(\rho)$ has the same pair of eigenvalues $\pm i \omega$ corresponding to the eigen\-space $V_0$ (see \eqref{eq:isotyp}) for all values of the parameter $\rho$, 
while the other seven pairs of complex conjugate eigenvalues of  $A(\rho)$ cross the imaginary axis transversely through the pair $\pm i\omega$ for $\rho=0$. 
This is a degenerate situation because the crossing number is not defined, however we can use Theorem \ref{thm:main-theorem}.


The complexification of the phase space is an $O_4\times S^1$-representation admitting the isotypical decomposition
$$
{}^1\widetilde {V}= {}^1\widetilde V_0\oplus {}^1\widetilde V_1 \oplus {}^1\widetilde V_2 \oplus {}^1\widetilde V_3,
$$ 
where $^1\widetilde V_k$ is modeled on the irreducible representation $^1\widetilde {\mathcal V_k}$ (recall that all the $O_4$-isotypical components of $V$ are modeled on irreducible representations of {\it real} type). In order to apply Theorem  \ref{thm:main-theorem}, one needs to choose maximal twisted subgroups $H^{\varphi}$ occurring in ${}^1\widetilde V_k$ with $k = 1,2,3$ such that 
\begin{equation}\label{eq:non-itersect-cond-D3}
^1\widetilde {\mathcal V_0} ^{H^\varphi} = \{0\}.
\end{equation}
It is easy to verify that with the exception of $D_3$, condition \eqref{eq:non-itersect-cond-D3} is satisfied for all maximal twisted subgroups.
Hence, 
Theorem \ref{thm:main-theorem} guarantees the existence of bifurcating branches of periodic solutions with symmetries $(^-\overline{D}_4^z),(^-\overline{D}_3^d)$, 
$(^-\overline{D}_2^d),(^-\overline{\mathbb Z}_4^c),(^-\overline{\mathbb Z}_3^t), (^+\overline{D}_4^d), (^+\overline{D}_2^d),(^+\overline{\mathbb Z}_4^c),
(^+\overline{\mathbb Z}_3^t),(^-\overline{S}_4^-)$. All these branches bifurcate from the trivial solution at the bifurcation point $\rho=0$.

\section{Appendix}\label{appendix}
Given a cube with subsequent vertices $1,2,3,4$ on one facet, and subsequent vertices $5,6,7,8$ on the opposite facet, with the vertices $1$ and $5$ connected by an edge, denote by $O_4$ the subgroup of the symmetry group $S_8$ consisting of all symmetries of the above cube, and by $S_4$ the subgroup of $O_4$ consisting of all symmetries of the cube preserving its orientation. As is well-known, $S_4$ can be thought of as the group of permutations of the large diagonals of the cube. Denote by $O_1$ a subgroup 
of $O_4$ generated by the permutation $(17)(28)(35)(46)$. Clearly, $O_1$ is isomorphic to $\mathbb Z_2$ and $O_4 = S_4 \times O_1$.

Define two subgroups $(\mathbb Z_2\times O_1  )^{o}, (\mathbb Z_2 \times O_1  )^{oz} < \mathbb Z_2 \times O_4 \times S^1=: G$ (both isomorphic to 
$\mathbb Z_2\times O_1$)  by
\begin{align*}
(\mathbb Z_2\times O_1  )^{o} :=& \Big\{\big(1,( ),0\big), 
 \big(1,(17)(28)(35)(46),0\big), \big(-1,( ),1/2\big),  \\&
 \big(-1,(17)(28)(35)(46),1/2\big)\Big\},
\end{align*}
 \begin{align*}
 (\mathbb Z_2 \times O_1  )^{oz} :=& \Big\{\big(1,( ),0\big), \big(-1,(17)(28)(35)(46),0\big), \big(-1,( ),1/2\big), \\& 
 \big(1,(17)(28)(35)(46),1/2\big)\Big\},
 \end{align*}
 and two their subgroups (both isomorphic to $O_1$):
 \begin{equation}
 (1_{\mathbb Z_2} \times O_1  )^{o}:=  \Big\{\big(1,( ),0\big),  \big(1,(17)(28)(35)(46),0\big)\Big\},
 \end{equation}
 
  \begin{equation}
 (1_{\mathbb Z_2} \times O_1  )^{oz}:=  \Big\{\big(1,( ),0\big),  \big(1,(17)(28)(35)(46),1/2\big)\Big\}
 \end{equation}
 (here $1_{\mathbb Z_2}$ stands for the neutral element in $\mathbb Z_2$).
 Then, for any $H < S_4 \times S^1 \simeq \{1\} \times S_4 \times S^1 < G$, define the subgroups  ${}^+H,\;  {}^-H,\; {}^+\overline{H}, \; {}^-\overline{H}  < G$ by
 \[
 {}^+H := H \cdot (\mathbb Z_2\times O_1)^{o}, \quad {}^-H := H \cdot (\mathbb Z_2\times O_1)^{oz},\; 
 \]
 \[
 {}^+\overline{H}:= H \cdot (1_{\mathbb Z_2}\times O_1)^{o}, \quad
  {}^-\overline{H} :=  H \cdot (1_{\mathbb Z_2}\times O_1)^{oz}.
  \]
Clearly, for any $H < S_4$, one has $H \cap (\mathbb Z_2\times O_1)^{o} = H \cap (\mathbb Z_2\times O_1)^{oz} = 1_G$, where $1_G$ stands for the neutral element in $G$, therefore ${}^+H$ and ${}^-H$ are isomorphic to the direct products of their factors. 
 All twisted subgroups of $G$ which we deal with in Example 2 appear as 
either $^+H$, or $^-H$, or $ {}^+\overline{H}$, or  ${}^-\overline{H}$, where $H$ is among the following groups: 
  \begin{align*}
 S_4:=&\Big\{\big((),0\big),\big((15)(28)(37)(46),0\big),\big((17)(26)(35)(48),0\big),\big((12)(35)(46)(78),0\big),\\&
 \big((17)(28)(34)(56),0\big),\big((14)(28)(35)(67), 0\big),\big((17)(23)(46)(58),0\big),\\&
\big((13)(24)(57)(68),0\big),\big((18)(27)(36)(45),0\big),\big((16)(25)(38)(47),0\big),\big((254)(368),0\big),\\&
 ((245)(386),0),((163)(457),0),((136)(475),0),((168)(274),0),\\&((186)(247),0),
 \big((138)(275),0\big),\big((183)(257),0\big),
 \big((1234)(5678),0\big),\big((1432)(5876),0\big),\\& (\big(1265)(3874),0\big),\big((1562)(3478),0\big),
 \big((1485)(2376),0\big),\big((1584)(2678),0\big)\Big\}\\
 D_4^z:=&\Big\{\big((),0\big),\big((1234)(5678),0\big),\big((13)(24)(57)(68),0\big),\big((1432)(5876),0\big),
 \\&\big((17)(26)(35)(48),1/2\big),\big((18)(27)(36)(45),1/2\big),\big((15)(28)(37)(46),1/2\big),\\&
 \big((16)(25)(38)(47),1/2\big)\Big\}\\
 D_3^z:=&\Big\{\big((),0\big),\big((254)(368),0\big),\big((245)(386),0\big),\big((17)(26)(35)(48),1/2\big),
 \\&\big((17)(28)(34)(56),1/2\big),\big((17)(23)(46)(58),1/2\big)\Big\}\\
 D_2^d:=&\Big\{\big((),0\big),\big((17)(26)(35)(48),1\big),\big((13)(24)(57)(68),1/2\big),\big((15)(28)(37)(46),
 1/2\big)\Big\}\\
 \mathbb Z^c_4:=&\Big\{\big((),0\big),\big((1234)(5678),1/4\big),\big((13)(24)(57)(68),1/2\big),\big((1432)(5876),3/4\big)\Big\}
 \\
 \mathbb{Z}^t_3:=&\Big\{\big((),0\big),\big((254)(368),1/3\big),\big((245)(386),2/3\big)\Big\}\\
 D_4^d:=&\Big\{\big((),0),((1234)(5678),1/2\big),\big((13)(24)(57)(68),0\big),\big((1432)(5876),1/2\big),
 \\&\big((17)(26)(35)(48),0\big),\big((18)(27)(36)(45),1/2\big),\big((15)(28)(37)(46),0\big), \\& \big((16)(25)(38)(47),1/2\big)\}\\
 D_3:=&\{\big((),0\big),\big((254)(368),0\big),\big((245)(386),(17)(26)(35)(48),0\big),\\&\big((17)(28)(34)(56),0\big),\big((17)(23)(46)(58),0\big)\Big\}\\
 \end{align*}
 \begin{align*}
 S_4^-:=&\Big\{\big((),0\big),\big((15)(28)(37)(46),1/2\big),\big((17)(26)(35)(48),1/2\big),\big((12)(35)(46)(78)
 ,1/2\big),\\&\big((17)(28)(34)(56),1/2\big),\big((14)(28)(35)(67)1/2\big),\big((17)(23)(46)(58),1/2\big), 
 \\&((13)(24)(57)(68),0),((18)(27)(36)(45),0),((16)(25)(38)(47),0),((254)(368),0)
 ,\\&\big((245)(386),0\big),\big((163)(457),0\big),\big((136)(475),0\big),\big((168)(274),0\big),\big((186)(247),0\big),\\&
 ((138)(275),0), ((183)(257),0),
 \big((1234)(5678),1/2\big),\big((1432)(5876),1/2\big),\\&\big((1265)(3874),1/2\big),\big((1562)(3478),1/2\big),
 \big((1485)(2376),1/2\big),\big((1584)(2678),1/2\big)\Big\}
 \end{align*}

\section*{Acknowledgments}
This work has been done as a part of the Prospective Human
Resources Support Program of the Czech Academy of Sciences; EH acknowledges the support by this program.

\bibliographystyle{abbrv}
\bibliography{the_bib}

\end{document}